\documentclass[a4paper]{amsart}

\usepackage[english]{babel}
\usepackage[utf8x]{inputenc}
\usepackage[T1]{fontenc}
\usepackage{amsthm}
\usepackage{amsmath}
\theoremstyle{plain}
\usepackage{chngcntr}
\usepackage{relsize}
\usepackage{pdfpages}
\usepackage{xcolor}
\usepackage[left,mathlines]{lineno}

\newtheorem{Lemma}{Lemma}
\newtheorem{Theorem}[Lemma]{Theorem}
\newtheorem{Proposition}[Lemma]{Proposition}
\newtheorem{Corollary}[Lemma]{Corollary}

\usepackage[a4paper,top=3cm,bottom=2cm,left=3cm,right=3cm,marginparwidth=1.75cm]{geometry}

\usepackage{float}
\usepackage{nccmath}
\usepackage{mathtools}
\usepackage{amsfonts}
\usepackage{amsmath}
\usepackage{graphicx}
\usepackage[colorinlistoftodos]{todonotes}
\usepackage[colorlinks=true, allcolors=blue]{hyperref}

\title{Diophantine approximation with primes from short intervals}



\author{Stephan~Baier}
\address{Stephan~Baier\\
	Ramakrishna Mission Vivekananda Educational Research Institute\\
	Department of Mathematics\\
	G.\ T.\ Road, PO~Belur Math, Howrah, West Bengal~711202\\
	India}
\email{stephanbaier2017@gmail.com}
\urladdr{https://www.researchgate.net/profile/Stephan\_Baier2}

\author{Sayantan Roy}
\address{Sayantan Roy\\
	Higher Education Directorate\\
	Bikash Bhavan\\
	Kolkata, West Bengal~700091\\
	India}
\email{roysayantan6@gmail.com}
\urladdr{}

\begin{document}
\maketitle

\begin{abstract} In this paper, we establish hybrid results on Diophantine approximation with primes from short intervals. In particular, we prove the following result in a slightly modified form: If $\alpha$ is an irrational number having a continued fraction expansion with bounded terms (in particular, if $\alpha$ is a quadratic irrational), then the number of primes $p$ in the interval $(X-Y,X]$ satisfying $||p\alpha||<\delta$ is asymptotically equal to $2\delta Y/\log X$, provided that $X\ge 10$,  
$X^{2/3+\varepsilon}\le Y\le X/2$ and $X^{\varepsilon}\max\left\{X^{1/4}Y^{-1/2},X^{2/3}Y^{-1}\right\}\le \delta\le 1/2$. 
\end{abstract}

\bigskip

\tableofcontents

\section{Introduction}
\subsection{Diophantine approximation with primes}
Let $\alpha\in \mathbb{R}$ be irrational. Improving a result of Vinogradov \cite{Vin}, Vaughan \cite{Vau} proved that there are infinitely many primes $p$ such that 
$$
||p\alpha||<p^{-1/4+\varepsilon},
$$
where for $x\in \mathbb{R}$, $||x||$ denotes the distance of $x$ to the nearest integer.  
In fact, his method gives more than that, namely the infinitude of natural numbers $X$ such that the asymptotic
$$
\sum\limits_{\substack{X/2<p\le X\\ ||p\alpha||<X^{-\sigma}}} 1\sim \frac{X^{1-\sigma}}{\log X}
$$
holds whenever $\sigma\in [0,1/4)$ is fixed. This result was improved by several authors, including Harman \cite{Har83}, \cite{Har96}, Heath-Brown and Jia \cite{HJ} and Matom\"aki \cite{Mat} who established the infinitude of primes $p$ 
such that 
$$
||p\alpha||<p^{-1/3+\varepsilon}. 
$$
Matom\"aki's method gives the infinitude of natural numbers $X$ such that a lower bound
$$
\sum\limits_{\substack{X/2<p\le X\\ ||p\alpha||<X^{-\sigma}}} 1\gg \frac{X^{1-\sigma}}{\log X}
$$
of expected order of magnitude holds whenever $\sigma\in [0,1/3)$ is fixed. 

\subsection{Primes in short intervals}
Let $\pi(X)$ be the prime counting function. Huxley \cite{Hux} proved the asymptotic estimate 
$$
\pi(X)-\pi(X-Y)\sim \frac{Y}{\log X} \quad \mbox{as } X\rightarrow \infty
$$
if $X^{7/12+\varepsilon}\le Y\le X$.  Under the Riemann Hypothesis for the zeta function, the exponent $7/12$ can be replaced by $1/2$. Huxley's exponent $7/12$ stood as a record until Guth and Maynard replaced it by $17/30$ in their recent groundbreaking paper \cite{GuMa}. This result is a corollary of their new zero density estimates for the Riemann zeta function. If one considers lower bounds of the correct order of magnitude in place of an asymptotic, then sieve methods allow for a smaller exponent. The record is due to Baker, Harman and Pintz who showed that       
$$
\pi(X)-\pi(X-Y)\gg \frac{Y}{\log X} \quad \mbox{as } X\rightarrow \infty
$$
if $X^{21/40+\varepsilon}\le Y\le X$.

\subsection{A hybrid result} We seek to evaluate the sums 
$$
\sum\limits_{X-Y<n\le X} 1_{\mathbb{P}}(n)1_{[0,\delta)}(||n\alpha||)
$$
asymptotically, where $1_{\mathbb{P}}$ and $1_{[0,\delta)}$ are the indicator functions of the set of primes $\mathbb{P}$ and the interval $[0,\delta)$, respectively. To simplify calculations, we make two adjustments: 1) As common in this context, we replace $1_{\mathbb{P}}(n)$ by the von Mangoldt function $\Lambda(n)$. 2) We replace $1_{[0,\delta)}(||x||)$ by a smooth 1-periodic function, defined as
\begin{equation} \label{Fdef}
F(x):=\sum\limits_{n\in \mathbb{Z}} \exp\left(-\pi\cdot \frac{(x-n)^2}{\delta^2}\right).
\end{equation}
This function imitates the function $1_{[0,\delta)}(||x||)$ since 
$$
\begin{cases} F(x) \gg 1 & \mbox{ if } ||x||\in [0,\delta),\\ 
F(x) \ll 1 & \mbox{ for all } x\in \mathbb{R},\\ 
F(x) \ll X^{-2025} & \mbox{ if } ||x||\in \left[X^{\varepsilon}\delta,1/2\right]. \end{cases} 
$$
We shall
establish the following hybrid result on Diophantine approximation with prime powers from short intervals $(X-Y,X]$ , where $Y$ lies in a range of the form $X^{2/3+\varepsilon}\le Y\le X/2$.
 
\begin{Theorem}\label{main} Suppose $\alpha$ is an irrational number, $\varepsilon>0$, and $X,Y,\delta\in \mathbb{R}$, $q\in \mathbb{N}$ satisfy the following properties: 
\begin{equation} \label{Dio}
\left| \alpha - \frac{a}{q} \right| < \frac{1}{q^2} \quad \mbox{ for a suitable integer } a \mbox{ with } (a,q)=1,
\end{equation}
\begin{equation} \label{XYdelta}
X\ge 10,\quad  X^{2/3+10\varepsilon}\le Y\le X/2, \quad X^{10\varepsilon}\max\left\{X^{1/4}Y^{-1/2},\ X^{2/3}Y^{-1}\right\}\le \delta\le 1/2
\end{equation}
and 
\begin{equation} \label{qcond}
\frac{Y}{\delta X^{1/2-2\varepsilon}}\le q\le \frac{Y}{\delta X^{1/2-3\varepsilon}}.
\end{equation}
Then, 
\begin{equation} \label{desired}
\sum\limits_{X-Y<n\le X} \Lambda(n)F(n\alpha) \sim \delta Y,
\end{equation}
as $q\rightarrow \infty$. 
\end{Theorem} 

We note that the condition "$q\rightarrow \infty$" above is meaningful since the continued fraction expansion of $\alpha$ yields infinitely many $q\in \mathbb{N}$ satisfying the relation \eqref{Dio}. 
 
We may remove the above condition \eqref{qcond} relating $q$ and $X,Y,\delta$ for a certain class of $\alpha$'s with "good" continued fraction expansions.  

\begin{Corollary} \label{firstCor} Suppose that $\alpha$ is irrational having a continued fraction expansion $\alpha=[a_0,a_1,...]$ with bounded terms $a_n$
(which is true for quadratic irrationals $\alpha$, in particular). Suppose that 
$$
X\ge 10,\quad   X^{2/3+10\varepsilon}\le Y\le X/2, \quad X^{10\varepsilon}\max\left\{X^{1/4}Y^{-1/2},\ X^{2/3}Y^{-1}\right\}\le \delta\le 1/2.
$$ 
Then
\begin{equation*}
\sum\limits_{X-Y<n\le X} \Lambda(n)F(n\alpha) \sim\delta Y,
\end{equation*}
as $X\rightarrow \infty$. 
\end{Corollary}

\begin{proof} The convergents of $\alpha$ are of the form $p_n/q_n$, where  $q_0=1$, $q_1=a_1$ and $q_n=a_nq_{n-1}+q_{n-2}$ for all $n\ge 2$. Let $M$ be such that $a_n\le M$ for all $n\in \mathbb{N}\cup \{0\}$. Then it follows that $q_{n-1}< q_n\le (M+1)q_{n-1}$ for all $n\in \mathbb{N}$. In this case, if $X$ is large enough, there exists $n\in \mathbb{N}$ such that
$$
\frac{Y}{\delta X^{1/2-2\varepsilon}}\le q_n\le \frac{Y}{\delta X^{1/2-3\varepsilon}}
$$
 and $q=q_n$ satisfies \eqref{Dio}.  Now the result follows from Theorem \ref{main}. 
\end{proof}

Of course, we could have stated Corollary \ref{firstCor} with $\varepsilon$ in place of $10\varepsilon$, but keeping the term $10\varepsilon$ makes it easier to write its proof.
 
As usual in this context, partial summation arguments allow us to pass from the sum in \eqref{desired} involving the von Mangoldt function to a corresponding sum over primes. Moreover, with some additional efforts on the Fourier-analytic side (e.g., using approximations by Beurling-Selberg polynomials), it is possible to derive similar results as above for the indicator function $1_{[0,\delta)}(||x||)$ in place of its smooth counterpart $F(x)$. We abstain from working out the details of these calculations but state below an analogue of Corollary \ref{firstCor} which we can get in this way.

\begin{Corollary} \label{secondCor} Suppose the conditions in Corollary \ref{firstCor} are satisfied. Then
\begin{equation*}
\sharp\{p\ \text{\rm prime} : X-Y<p\le X, \ ||p\alpha||<\delta\}\sim \frac{2\delta Y}{\log X}, 
\end{equation*}
as $X\rightarrow \infty$. 
\end{Corollary}

{\bf Acknowledgements.} The authors wish to tank the anonymous referees for their valuable comments and the Ramakrishna Mission for excellent working conditions. 

\section{Preliminaries}
Throughout, we use the following notations.

\begin{itemize}
\item $\mathbb{N}$ and $\mathbb{Z}$ denote the sets of natural numbers and integers, respectively. 
\item For $x\in \mathbb{R}$, we write $\exp(x)=e^x$ and $e(x)=e^{2\pi i x}$. 
\item $\Lambda(n)$ is the von Mangoldt function, defined as 
$$
\Lambda(n):=\begin{cases} \log p & \mbox{ if } n=p^k \mbox{ with } p \mbox{ prime and } k\in \mathbb{N},\\ 0 & \mbox{ otherwise.}\end{cases} 
$$
\item $\tau(n)$ denotes the number of positive divisors of the natural number $n$.
\item $||x||$ stands for the distance of $x\in \mathbb{R}$ to the nearest integer.
\item We write $f(x)\ll g(x)$ or $f(x)=O(g(x))$ if $|f(x)|\le cg(x)$ for a suitable constant $c>0$ and large enough $x$. 
\item We write $f_1(x)=f_2(x)+O(g(x))$ if $f_1(x)-f_2(x)\ll g(x)$. 
\item We write $f(x)\sim g(x)$ as $x\rightarrow \infty$ if $\lim\limits_{x \rightarrow \infty} f(x)/g(x)=1$. 
\item If $a,b\in \mathbb{Z}$ are not both equal to $0$, then $(a,b)$ denotes the greatest common divisor of $a$ and $b$.  
\item As usual, $\varepsilon$ denotes an arbitrarily small but fixed positive real number. 
\end{itemize}

We shall use Vaughan's identity to relate sums involving the von Mangoldt function to bilinear sums. 

{\begin{Proposition}[Vaughan's identity]\label{Prop1}
Suppose that $U,V\ge 1$ and $UV\le X$. Then for every arithmetic function $f:\mathbb{N}\rightarrow \mathbb{C}$, we have
$$
\sum\limits_{U<n\le X} \Lambda(n) f(n)\ll (\log 2X)S_1+S_2
$$   
with
\begin{equation} \label{type I}
	S_1:=\sum\limits_{m\le UV} \max\limits_{w} \Bigg|\sum\limits_{w<n\le X/m} f(mn)\Bigg|
\end{equation}
and 
\begin{equation} \label{type II}
	S_2:=\Bigg|\sum_{U<m\le X/V}\ \sum\limits_{V<n\le X/m} \Lambda(m)b(n)f(mn)\Bigg|,
\end{equation}
where $b(n)\in \mathbb{R}$ depends only on $V$ and satisfies $|b_n|\le \tau(n)$.  
\end{Proposition}

\begin{proof}
This is \cite[Satz 6.1.2]{Bru}.
\end{proof}

It is customary to refer to the sums in equation \eqref{type I} as type I bilinear sums and to the sums in equation \eqref{type II} as type II bilinear sums. 

The following result provides an approximation of the function $F(x)$, defined in \eqref{Fdef}, by trigonometrical polynomials.   

\begin{Proposition}[Approximation by trigonometrical polynomials] \label{Lemma:TS}
 Let $0< \delta \leq 1/2$ and $L\ge X^{\varepsilon}\delta^{-1}$. Then, uniformly for all $x\in \mathbb{R}$, we have 
$$ 
F(x)=\delta+\delta \sum\limits_{0<|\ell|\le L} \exp\left(-\pi \delta^2 \ell^2\right)e(\ell x)+O\left(X^{-2025}\right).
$$
\end{Proposition}

\begin{proof} Poisson summation yields
$$
F(x)=\delta \sum\limits_{\ell \in \mathbb{Z}} \exp\left(-\pi \delta^2 \ell^2\right)e(\ell x).
$$
Now the result follows from 
$$
\left|\sum\limits_{|\ell|>L} \exp\left(-\pi \delta^2 \ell^2\right)e(\ell x)\right|\le  2\sum\limits_{\ell> X^{\varepsilon}\delta^{-1}} \exp\left(-\pi \delta^2 \ell^2\right)\ll X^{-2025}. 
$$
\end{proof}

We will require the following fundamental result in Diophantine approximation.  

\begin{Proposition}[Rational approximation] 
Let $\alpha\in \mathbb{R}$ be an irrational number. Then there exist infinitely many $q\in \mathbb{N}$ such that  
\[
\left| \alpha - \frac{a}{q} \right| < \frac{1}{q^2} \quad \mbox{for a suitable integer } a \mbox{ with } (a,q)=1. \]
\end{Proposition}

\begin{proof} The convergents $a/q=p_n/q_n$ of the continued fraction expansion of $\alpha$ satisfy the above property, and there are infinitely many of them. Alternatively, this result can be derived from the Dirichlet approximation theorem. \end{proof}

The following basic bound for linear exponential sums will be crucial. 

\begin{Proposition}[Bound for linear exponential sums] \label{expbound} Let $\alpha\in \mathbb{R}$ and $I$ be an interval of length $\mu(I)$. Then 
$$
\left|\sum\limits_{n\in I\cap \mathbb{N}} e(nx)\right|\ll \min\left\{\mu(I)+1,\frac{1}{||x||}\right\}.
$$ 
\end{Proposition}

\begin{proof} The upper bound by $\mu(I)+1$ is trivial. The upper bound by $1/||x||$ can be found in \cite{GrKo}[page 6].  
\end{proof}

Applying the above Proposition \ref{expbound} to bilinear sums with exponentials will result in sums of the form 
$$
\Sigma(M,N)=\sum_{1\le m\leq M}\min\left\{N, \dfrac{1}{||\alpha m||}\right\},
$$  
for which we have the following standard bound depending on rational approximation of $\alpha$. 

\begin{Proposition}[Standard estimate] \label{standard}
Suppose that $M, N\ge 1$ and $|\alpha-a/q|<q^{-2}$, where $q\in \mathbb{N}$ and $a\in\mathbb{Z}$ with $(a,q)=1$. Then 
\begin{equation*} 
\begin{split}
\sum_{1\le m\leq M}\min\left\{N, \dfrac{1}{||\alpha m||}\right\} \ll & \begin{cases}  MNq^{-1}+M\log(MNq) & \mbox{ if } M> q/2\\ q\log q & \mbox{ if } M\le q/2, \end{cases}\\
\ll & MNq^{-1}+(M+q)\log(MNq).
\end{split}
\end{equation*}
\end{Proposition}

\begin{proof} The result is trivial if $q\le 2$. So let $q> 2$. If $M\le q/2$, then 
$$
\left|m\alpha-\frac{ma}{q}\right|\le \frac{1}{2q}.
$$ 
It follows that 
$$
\sum_{1\le m\leq M}\min\left\{N, \dfrac{1}{||\alpha m||}\right\}\ll \sum_{1\le m\leq M} \dfrac{1}{\left|\left|\frac{am}{q}\right|\right|}\le \sum_{1\le n\leq M} \dfrac{1}{\left|\left|\frac{n}{q}\right|\right|} = \sum_{1\le n\leq M} \frac{q}{n}\ll q\log q
$$
since multiplication by $a$ permutes residue classes modulo $q$, and $am\not\equiv 0\bmod{q}$ if $1\le m\le M$. If $M> q/2$, then the result follows from \cite[Lemma 2.2]{Vau2}.  
\end{proof}

Finally, we will use the following result about primes in short intervals by Guth and Maynard.

\begin{Proposition}[Guth and Maynard] \label{GM} As $X\rightarrow\infty$, we have   
$$
\sum\limits_{X-Y<n\le X} \Lambda(n) \sim Y,
$$
provided that $X^{17/30+\varepsilon}\le Y\le X/2$. 
\end{Proposition}

\begin{proof} This is a consequence of \cite{GuMa}[Corollary 1.3]. \end{proof}

In fact, Huxley's weaker result in \cite{Hux} with the exponent $7/12$ in place of $17/30$ would suffice for our purposes since $7/12<2/3$. 

\section{Basic steps}
Throughout the sequel, assume that the conditions in Theorem \ref{main} are satisfied. We start by writing
$$
\sum\limits_{X-Y<n\le X} \Lambda(n)F(n\alpha) = \delta \sum\limits_{X-Y<n\le X} \Lambda(n)+
\sum\limits_{X-Y<n\le X} \Lambda(n)f(n),
$$
where 
$$
f(n):=F(n\alpha)-\delta.
$$
In view of Proposition \ref{GM}, it suffices to prove that 
\begin{equation} \label{fina}
\sum\limits_{n\le X} \Lambda(n)f(n)\ll \delta YX^{-\eta}
\end{equation}
for some $\eta>0$. We use Vaughan's identity with 
$$
U=X^{1/3}=V
$$
to bound the above sum by 
$$
\sum\limits_{X-Y<n\le X} \Lambda(n)f(n) \ll (\log X)S_1+S_2
$$   
with
\begin{equation*} 
	S_1:=\sum\limits_{m\le X^{2/3}} \max\limits_{(X-Y)/m\le w\le X} \Bigg|\sum\limits_{w<n\le X/m} f(mn)
\Bigg|
\end{equation*}
and 
\begin{equation*} 
	S_2:=\Bigg|\sum\limits_{X^{1/3}<m\le X^{2/3}}\ \sum\limits_{\max\left\{X^{1/3},(X-Y)/m\right\}<n\le X/m} \Lambda(m)b(n)f(mn)\Bigg|.
\end{equation*}
Throughout the sequel, set
$$
L:=X^{\varepsilon}\delta^{-1}.
$$
By Proposition \ref{Lemma:TS}, it is then enough to prove that 
\begin{equation} \label{thegoal}
S_1', S_2'\ll YX^{-\eta}
\end{equation}
for some $\eta>0$, where 
\begin{equation*} 
	S_1':=\sum\limits_{m\le X^{2/3}} \max\limits_{(X-Y)/m\le w\le X} \Bigg|\sum\limits_{w<n\le X/m}\sum_{0<\ell \leq L}\ c(\ell)e(\ell mn\alpha) \Bigg|
\end{equation*}
and 
\begin{equation*} 
	S_2':=\sum_{X^{1/3}<m\le X^{2/3}} \sum\limits_{\max\left\{X^{1/3},(X-Y)/m\right\}<n\le X/m} \Lambda(m)b(n)\sum_{0<\ell \leq L}c(\ell)e(\ell mn \alpha)
\end{equation*}
with 
$$
c(\ell):=\exp\left(-\pi \delta^2\ell^2\right).
$$
After splitting the $\ell$-sums into $O(\log X)$ dyadic subsums and rearranging summations, it remains to prove that whenever 
\begin{equation*}
1\le H\le L, \quad |c(h)|\le 1, \quad 0\le |b(n)|\le \tau(n),
\end{equation*}
the sums 
\begin{equation} \label{T1sum}
T_1(H):=\sum_{H/2<h \le H} |c(h)| \sum_{m\leq X^{2/3}} \max_{(X-Y)/m\le w\le X/m} \Bigg|\sum\limits_{w<n\le X/m}e\left(hmn\alpha \right)\Bigg|
\end{equation}
and 
\begin{equation} \label{T2sum}
T_2(H):= \sum_{H/2<h\le H}c(h)\sum_{X^{1/3}<m\leq X^{2/3}}\ \sum\limits_{\max\left\{X^{1/3},(X-Y)/m\right\}<m\le X/m}\Lambda(m)b(n) e\left(hmn\alpha \right)
\end{equation}
satisfy bounds of the form 
\begin{equation} \label{ult}
T_1(H),T_2(H)\ll YX^{-\eta}
\end{equation}
for some $\eta>0$. 

\section{Estimation of the type I sum} \label{Type I section}
In this section, we bound the type I sum $T_1(H)$, defined in \eqref{T1sum}. We apply Proposition \ref{expbound} to estimate the smooth innermost sum over $n$ by
$$
\sum_{w<n\le {X}/{m}} e\left(hmn\alpha \right)\ll \min\left\{\dfrac{Y}{m}, \dfrac{1}{||hm\alpha||} \right\} \quad \mbox{if } (X-Y)/m\le w,
$$
getting
$$
T_1(H)\ll\sum_{H/2<h\le H} \sum_{m\leq X^{2/3}}  \min\left\{\dfrac{Y}{m}, \dfrac{1}{||hm\alpha||} \right\}.
$$
Splitting the $m$-summation on the right-hand side into dyadic subsums, it now suffices to establish that
$$
\sum_{H/2<h\le H}\ \sum_{M/2<m\le M} \min\left\{\dfrac{Y}{M}, \dfrac{1}{||hm\alpha||} \right\} \ll YX^{-\eta} \quad \mbox{ if } 1\le M\le X^{2/3},
$$
for some $\eta>0$. Combining the variables $h$ and $m$ into a single variable $k=hm$, it remains to show that 
\begin{flalign} \label{remains}
\sum_{k\le MH} \min\left\{\dfrac{Y}{M}, \dfrac{1}{||k\alpha||} \right\}\ll YX^{-\eta},
\end{flalign}
for some $\eta>0$. Proposition \ref{standard} yields 
$$
\sum_{k\le MH} \min\left\{\dfrac{Y}{M}, \dfrac{1}{||k\alpha||} \right\}\ll \begin{cases} 
HYq^{-1}+MH\log(HYq)  & \mbox{ if } MH> q/2,\\
q\log q & \mbox{ if } MH\le q/2.\end{cases}
$$
It follows that if $1\le H\le L=X^{\varepsilon}\delta^{-1}$ and $1\le M\le X^{2/3}$, then 
$$
\sum_{k\le MH} \min\left\{\dfrac{Y}{M}, \dfrac{1}{||k\alpha||} \right\}\ll 
\begin{cases} \delta^{-1}X^{\varepsilon}Yq^{-1}+\delta^{-1}X^{2/3+\varepsilon}\log(HYq)  & \mbox{ if } MH> q/2,\\
q\log q & \mbox{ if } MH\le q/2.\end{cases}
$$
Hence, \eqref{remains} holds if 
\begin{equation} \label{ourfirstcond}
\max\left\{Yq^{-1}, \ X^{2/3}, \ \delta q\right\}\le \delta YX^{-\eta-2\varepsilon}.
\end{equation}
This inequality is easily verified for $\eta<\varepsilon$ under the conditions of Theorem \ref{main}.

\section{Estimation of the type II sum}
In this section, we bound the type II sum $T_2(H)$, defined in \eqref{T2sum}. Splitting the $m$-summation into $O(\log X)$ dyadic subsums results in sums of the form
$$
    T_2(H,M):= \sum_{M/2<m\le M} \
    \sum_{\max\left\{X^{1/3},(X-Y)/m\right\}<n\le {X}/{m}}  a(m)b(n)\sum_{H/2<h\le H} c(h) e\left(hmn\alpha \right),
$$
where $|a(m)|\le \log m$ and $|b(n)|\le \tau(n)$.  
It now suffices to show that 
\begin{equation} \label{T2goal}
T_2(H,M) \ll YX^{-\eta} \quad \mbox{ if } 1\le H\le L \mbox{ and } X^{1/3}\le M\le X^{2/3},
\end{equation}
for some $\eta>0$. In fact, due to the symmetry of the sum $T_2(H,M)$, it suffices to consider the case when 
\begin{equation} \label{T2assump}
X^{1/2}\le M\le X^{2/3}
\end{equation}
since in the other case when $X^{1/3}\le M\le X^{2/3}$, we may reverse the roles of $m$ and $n$ in the process below. Thus we will assume \eqref{T2assump} throughout the following. 
Applying the Cauchy-Schwarz inequality, we have
\begin{flalign}\label{T2T3}
T_2(H,M)^2 \ll \Bigg(\sum_{M/2<m\le M}a(m)^{2} \Bigg){T}_{3}(H,M)\ll X^{\varepsilon}MT_3(H,M),
\end{flalign}
where
\begin{equation} \label{ad}
\begin{split}
T_3(H,M):=& \sum_{M/2<m\le M} \left| \sum\limits_{\max\left\{X^{1/3},(X-Y)/m\right\}<n\le X/m} b(n) \sum\limits_{H/2<h\le H} c(h)\right|^2\\
= & \sum_{M/2<m\le M}\ \sum_{\max\left\{X^{1/3},(X-Y)/m\right\}<n_1,n_2\le X/m} b(n_1)\overline{b(n_2)}\times\\ & \sum_{H/2<h_1,h_2\le H}c(h_1)\overline{c(h_2)} e\left(m\left(h_{1}n_{1}-h_{2}n_{2} \right) \alpha\right).
\end{split}
\end{equation}
Rearranging summations, we get 
\begin{equation} \label{T3split}
\begin{split}
T_3(H,M)= & \sum_{\max\left\{X^{1/3},(X-Y)/M\right\}<n_1,n_2\le 2X/M}b(n_1)\overline{b(n_2)} \sum_{H/2<h_1,h_2\le H}c(h_1)\overline{c(h_2)} \times\\ & \sum_{\max\{M/2,(X-Y)/n_{1}, (X-Y)/n_{2}\}<m\leq \min\{M,X/n_{1}, X/n_{2}\} }e\left(m\left(h_{1}n_{1}-h_{2}n_{2} \right) \alpha\right)\\
= & T_4(H,M)+T_5(H,M),
\end{split}
\end{equation}
where $T_4(H,M)$ is the contribution of $n_1\le n_2$ to $T_3(H,M)$, and $T_5(H,M)$ is the contribution of $n_1>n_2$ to $T_3(H,M)$, i.e.,
\begin{equation} \label{T4}
\begin{split}
T_4(H,M):= & \sum_{\max\left\{X^{1/3},(X-Y)/M\right\}<n_1\le n_2\le 2X/M}b(n_1)\overline{b(n_2)} \sum_{H/2<h_1,h_2\le H}c(h_1)\overline{c(h_2)}\times\\ &  \sum_{\max\{M/2,(X-Y)/n_{1}\}<m\leq \min\{M,X/n_{2}\}}e\left(m\left(h_{1}n_{1}-h_{2}n_{2} \right) \alpha\right)
\end{split}
\end{equation}
and 
\begin{equation*}
\begin{split}
T_5(H,M):= & \sum_{\max\left\{X^{1/3},(X-Y)/M\right\}<n_2< n_1\le 2X/M}b(n_1)\overline{b(n_2)} \sum_{H/2<h_1,h_2\le H}c(h_1)\overline{c(h_2)} \times\\ & \sum_{\max\{M/2,(X-Y)/n_{2}\}<m\leq \min\{M,X/n_{1}\}}e\left(m\left(h_{1}n_{1}-h_{2}n_{2} \right) \alpha\right).
\end{split}
\end{equation*}
In the following, we bound only the sum $T_4(H,M)$, the estimation of $T_5(H,M)$ being essentially the same. 

The $m$-sum in \eqref{T4} is empty unless
$$
\dfrac{X-Y}{n_{1}}< \dfrac{X}{n_{2}},
$$
which is equivalent to 
$$
\left(n_{2}-n_{1}\right)X< n_{2}Y.
$$ 
This together with $n_2\le 2X/M$ implies that 
\begin{equation} \label{n1n2}
n_{2}-n_{1} \leq \dfrac{2Y}{M}.
\end{equation}

Observing that
$$
\min\left\{M,\frac{X}{n_{2}}\right\}-\max\left\{\frac{M}{2},\frac{X-Y}{n_{1}}\right\}\le  \frac{X}{n_2}-\frac{X-Y}{n_1}\le \frac{Y}{n_1}\le \frac{2MY}{X}  
$$
if
$$
\frac{X}{2M}\le \frac{X-Y}{M}<n_1\le  n_2\le \frac{2X}{M} \quad (\mbox{recalling } Y\le X/2),
$$
we use Proposition \ref{expbound} to estimate the $m$-sum by 
$$
 \sum_{\max\{M/2,(X-Y)/n_{1}\}<m\leq \min\{M,X/n_{2}\}}e\left(m\left(h_{1}n_{1}-h_{2}n_{2} \right) \alpha\right)\ll
\min\left\{\dfrac{MY}{X}+1, \dfrac{1}{||\alpha\left( n_{1}h_{1}-n_{2}h_{2}\right)||}\right\}.
$$ 
Under the condition
\begin{equation} \label{anothercond}
\frac{MY}{X}\ge 1,
\end{equation}
it follows that 
\begin{equation*}
    T_{4}(H,M)\ll  X^{\varepsilon}\sum_{\substack{X/(2M)<n_{1}\leq n_{2}\leq 2{X}/{M}\\ n_{2}-n_{1}\leq {2Y}/{M}}}\ \sum_{H/2<h_1,h_2\le H} \min\left\{\dfrac{MY}{X}, \dfrac{1}{||\alpha\left( n_{1}h_{1}-n_{2}h_{2}\right)||}\right\},
\end{equation*}
where we have used \eqref{n1n2},  $Y\le X/2$, $|b(n)|\le \tau(n)\ll n^{\varepsilon}$ and $|c(h)|\le 1$ for $H/2<h\le H$.  Setting 
$$
n_{1}h_{1}-n_{2}h_{2}=l,
$$ 
we deduce that 
\begin{equation} \label{T4esti}
T_4(H,N)\ll X^{\varepsilon}\sum\limits_{|l|\le 2XH/M} \gamma(l)\cdot  \min\left\{\dfrac{MY}{X}, \dfrac{1}{||\alpha l||}\right\},
\end{equation}
where
$$
\gamma(l):=\sharp\left\{(n_{1},n_2,h_{1},h_{2})\in \mathbb{Z}^{4}:\dfrac{X}{2M}<n_{1}\le n_2\leq \dfrac{2X}{M},\ n_2-n_1\leq \dfrac{2Y}{M},\ H/2<h_1,h_2\le H,\ l=n_{1}h_{1}-n_2h_{2} \right\}.
$$
Setting $n_2-n_1=k$ and $h_1-h_2=s$, we have 
$$
l=n_{1}h_{1}-n_2h_{2}=n_{1}h_{1}-\left( n_{1}+k\right)h_{2}=n_1s-kh_2,
$$
and hence  
\begin{equation*}
\gamma(l)\le \left\{(n_{1},h_{2},k,s)\in \mathbb{Z}^{4}:\dfrac{X}{2M}<n_{1}\leq \dfrac{2X}{M},\ H/2<h_2\le H, \ 0\le k\leq \dfrac{2Y}{M}, |s| \leq \frac{H}{2},\ l=n_{1}s-kh_{2} \right\}.
\end{equation*}

Write
$$
\gamma(l)=\gamma_0(l)+\gamma_1(l),
$$
where $\gamma_0(l)$ is the contribution of $l+kh_2=0$ and $\gamma_1(l)$ is the contribution of $l+kh_2\not=0$. 
Clearly,
$$
\gamma_1(l)\le \sum_{\substack{0\le k\leq 2Y/M\\ H/2<h_2\le H\\ l+kh_2\not=0}} \tau\left(|l+kh_{2}| \right)\ll X^{\varepsilon}\cdot \frac{YH}{M},
$$ 
where we have used the fact that $M\le X^{2/3}\le Y$. 
In the following, we bound the contribution $\gamma_0(l)$ for the case when $l+kh_2=0$. Clearly,
$$
\gamma_0(l)=0 \quad \mbox{ if } l\not\in \left[-\frac{2YH}{M},0\right].
$$
If $l=0$, then $n_1s=l+kh_2=0$ forces $k=s=0$, and thus we have
$$
\gamma_0(0)\le \frac{2XH}{M}. 
$$ 
For the remaining cases, we get
$$
\gamma_0(l)\le \frac{2X}{M}\cdot \tau(|l|) \ll \frac{X^{1+\varepsilon}}{M} \quad \mbox{ if } l\in \left[-\frac{2YH}{M},-1\right]
$$
since $n_1s=l+kh_2=0$ forces $l=-kh_2$ and $s=0$. Recalling \eqref{T4esti}, it follows that 
\begin{equation}\label{T4e}
T_4(H,N)\ll X^{\varepsilon}\left(A+B+C\right),
\end{equation}
where 
$$
A:=X^{\varepsilon}\cdot \frac{YH}{M}\cdot \sum\limits_{|l|\le 2XH/M} \min\left\{\dfrac{MY}{X}, \dfrac{1}{||\alpha l||}\right\},
$$
\begin{equation} \label{U1e}
B:=\frac{2XH}{M}\cdot \dfrac{MY}{X}=2YH
\end{equation}
and 
$$
C:=\frac{X^{1+\varepsilon}}{M}\cdot\sum\limits_{1\le l\le 2YH/M} \min\left\{\dfrac{MY}{X}, \dfrac{1}{||\alpha l||}\right\},
$$
where we note that $||\alpha l||=||-\alpha l||$. 

Using Proposition \ref{standard} if $l\not=0$, we see that
\begin{equation} \label{U0e}
\begin{split}
A\ll & X^{\varepsilon}\cdot \frac{YH}{M}\cdot \left(\frac{YH}{q}+\frac{XH\log(HYq)}{M}+q\log(HYq)+\frac{MY}{X}\right)\\ 
\ll & X^{2\varepsilon}\left(\frac{Y^2H^2}{Mq}+\frac{XYH^2}{M^2}+\frac{YHq}{M}+\frac{HY^2}{X}\right)
\end{split}
\end{equation}
and 
\begin{equation}\label{U2e}
C\ll \frac{X^{1+2\varepsilon}q}{M},
\end{equation}
provided that 
\begin{equation} \label{newcondi}
\frac{2YH}{M}\le \frac{q}{2}.
\end{equation}

Combining \eqref{T4e}, \eqref{U1e}, \eqref{U0e} and \eqref{U2e}, and noting that $HY^2/X\le B=2YH$, we obtain 
\begin{equation} \label{T4all}
T_4(H,N)\ll X^{3\varepsilon}
\left(\frac{Y^2H^2}{Mq}+\frac{XYH^2}{M^2}+\frac{YHq}{M}+YH+\frac{Xq}{M}\right)
\end{equation}
under the conditions \eqref{anothercond} and \eqref{newcondi}.
In a similar way, we get the same bound for $T_5(H,N)$. Hence, using \eqref{T2T3} and \eqref{T3split}, we have
\begin{equation} \label{T2bound1}
T_2(H,M)^2\ll X^{4\varepsilon} \left(\frac{Y^2H^2}{q}+\frac{XYH^2}{M}+YHq+YHM+Xq\right)
\end{equation}
under the conditions \eqref{anothercond} and \eqref{newcondi}. Recall that we assumed 
$X^{1/2}< M\le X^{2/3}$ in \eqref{T2assump}. The conditions \eqref{anothercond} and \eqref{newcondi} hold in this case, provided that
\begin{equation} \label{transcond}
Y\ge X^{1/2} \quad \mbox{and} \quad q\ge \frac{4YH}{X^{1/2}}.
\end{equation}
Using \eqref{T2assump} and $H\le L=X^{\varepsilon}\delta^{-1}$, it follows that
$$
T_2(H,M)^2\ll X^{6\varepsilon}\left(\frac{Y^2}{\delta^2 q}+\frac{X^{1/2}Y}{\delta^2}+\frac{Yq}{\delta}+\frac{X^{2/3}Y}{\delta}+Xq\right).
$$
Thus, \eqref{T2goal} holds if 
\begin{equation} \label{finalcondis}
\max\left\{\frac{Y^2}{q}, \ X^{1/2}Y, \ \delta Yq, \ \delta X^{2/3}Y,\ \delta^2Xq\right\} \le\ \delta^2 Y^2X^{-\eta-6\varepsilon}. 
\end{equation}
Now we assume that 
\begin{equation} \label{qc}
\frac{Y}{\delta X^{1/2-2\varepsilon}}\le q\le \frac{Y}{\delta X^{1/2-3\varepsilon}},
\end{equation}
in accordance with Theorem \ref{main}. Then the second condition in \eqref{transcond} holds if $X$ is large enough. The first condition in \eqref{transcond} holds by the assumptions in Theorem \ref{main}. Moreover, if \eqref{qc} is satisfies, then \eqref{finalcondis} holds, provided that
$$
\max\left\{\delta X^{1/2}Y, \ X^{1/2}Y, \ X^{-1/2}Y^2, \  \delta X^{2/3}Y\right\} \le\ \delta^2 Y^2X^{-\eta-9\varepsilon}. 
$$
This is satisfied if $\eta<\varepsilon$ and
$$
\delta\ge X^{10\varepsilon}\min\left\{X^{1/2}Y^{-1},\ X^{1/4}Y^{-1/2},\ X^{-1/4},\ X^{2/3}Y^{-1}\right\},  
$$
which holds under the conditions of Theorem \ref{main}. This completes the proof of Theorem \ref{main}.\medskip\\  
{\bf Comment 1:} Keeping the variable $h$ inside the modulus square in \eqref{ad} turns out to be advantageous.  Moving it outside the modulus square using Cauchy-Schwarz would simplify the calculations but inflate the diagonal contribution by a factor of $H$, resulting in a stronger condition on $\delta$, compared to the condition in \eqref{XYdelta}. 
 \medskip\\
{\bf Comment 2:} Our conditions on $Y$ and $\delta$ depend on three parameters $\gamma$, $\kappa$ and $\lambda$ controlling the $m$-summation interval $m\le X^{\gamma}$ in the type I sum and the $m$-summation interval $X^{\kappa}<m\le X^{\lambda}$ in the type II sum. The choice in our present paper is $\gamma=2/3$, $\kappa=1/3$ and $\lambda=2/3$. The term $X^{1/4}Y^{-1/2}$ in the condition on $\delta$ in \eqref{XYdelta} can be improved if we are able to choose $\gamma<3/4$ and $1/4<\kappa<\lambda<1/2$ or $1/2<\kappa<\lambda <3/4$. This is possible (obtaining a lower bound instead of an asymptotic) if we can make Harman's sieve method developed in \cite{Har83} and \cite{Har96} in the context of the $p\alpha$-problem work for short intervals. The $Y$-range in \eqref{XYdelta} can be enlarged if we are able to choose $\gamma<2/3$ and $1/3<\kappa<\lambda<2/3$ in our problem. For a simultaneous improvement in both aspects, we need a choice of $\gamma<2/3$ and $1/3<\kappa<\lambda<1/2$ or $1/2<\kappa<\lambda<2/3$.

\end{document}